\documentclass[a4paper, 11pt]{amsart}
\title[{\tiny perverse coherent sheaves on blow-ups}]
{birational geometry of moduli spaces of perverse coherent sheaves on blow-ups}
\date{}
\author{Naoki Koseki}

\makeatletter
 
  \@addtoreset{equation}{section}
 \makeatother

\usepackage{amsmath, amssymb, amsthm, amscd}
\usepackage[frame,cmtip,curve,arrow,matrix,line,graph]{xy}
\usepackage{comment}

\theoremstyle{plain}
\newtheorem{thm}{Theorem}[section]
\newtheorem{prop}[thm]{Proposition}
\newtheorem{lem}[thm]{Lemma}

\newtheorem*{thm*}{Theorem}

\theoremstyle{definition}
\newtheorem{defin}[thm]{Definition}

\newtheorem*{NaC}{Notation and Convention}
\newtheorem*{ACK}{Acknowledgement}

\theoremstyle{remark}
\newtheorem{rmk}[thm]{Remark}

\newtheorem{ex}[thm]{Example}

\DeclareMathOperator{\ch}{ch}

\DeclareMathOperator{\rk}{rk}
\DeclareMathOperator{\Per}{Per}

\DeclareMathOperator{\id}{id}
\newcommand{\dR}{\mathbf{R}}
\newcommand{\dL}{\mathbf{L}}
\newcommand{\bP}{\mathbb{P}}
\newcommand{\bC}{\mathbb{C}}
\newcommand{\bZ}{\mathbb{Z}}
\newcommand{\mcO}{\mathcal{O}}
\newcommand{\mcE}{\mathcal{E}}

\newcommand{\mcH}{\mathcal{H}}
\newcommand{\mcN}{\mathcal{N}}
\newcommand{\mcS}{\mathcal{S}}
\newcommand{\mcQ}{\mathcal{Q}}
\newcommand{\mcV}{\mathcal{V}}
\newcommand{\mcW}{\mathcal{W}}
\newcommand{\mcT}{\mathcal{T}}
\newcommand{\mcP}{\mathcal{P}}
\newcommand{\hatP}{\hat{\mathbb{P}}}
\newcommand{\hatC}{\hat{\mathbb{C}}}
\DeclareMathOperator{\Hom}{Hom}
\DeclareMathOperator{\Tot}{Tot}
\DeclareMathOperator{\Coh}{Coh}

\DeclareMathOperator{\ext}{ext}
\DeclareMathOperator{\Ext}{Ext}
\DeclareMathOperator{\End}{End}
\DeclareMathOperator{\Image}{Image}
\DeclareMathOperator{\Hilb}{Hilb}
\DeclareMathOperator{\Sym}{Sym}

\DeclareMathOperator{\codim}{codim}
\DeclareMathOperator{\GL}{GL}

\DeclareMathOperator{\Bl}{Bl}

\DeclareMathOperator{\Exc}{Exc}
\DeclareMathOperator{\pt}{pt}
\DeclareMathOperator{\Gr}{Gr}

\begin{document}
\begin{abstract}
In order to study the wall-crossing formula 
of Donaldson type invariants on the blown-up plane, 
Nakajima-Yoshioka constructed 
a sequence of blow-up/blow-down diagrams connecting 
the moduli space of torsion free framed sheaves on projective plane, 
and that on its blow-up. 
In this paper, we prove that 
Nakajima-Yoshioka's diagram 
realizes the minimal model program. 
Furthermore, we obtain a fully-faithful embedding 
between the derived categories of these moduli spaces. 
\end{abstract}

\maketitle

\setcounter{tocdepth}{1}
\tableofcontents

\section{Introduction}
\subsection{Main result}
The wall crossing formulas of 
Donaldson-type invariants 
have been investigated 
in various papers. 
For example, the behavior of Donaldson invariants 
of the moduli spaces of rank two stable sheaves 
on rational surfaces under the variations of 
polarizations are studied by 
Ellingsrud-G\"{o}ttsche \cite{eg95}, 
Friedman-Qin \cite{fq95}. 
As another example, 
Nakajima-Yoshioka \cite{ny11a, ny11b, ny11c} 
studied the difference of invariants 
of framed sheaves on the projective plane $\bP^2$ 
and that on its blow-up $\hatP^2$ at a point. 
To do so, they constructed a sequence of diagrams 
\begin{equation} \label{eq:nyintro}
\xymatrix{
&\cdots &M^m(v) \ar[ld] \ar[rd]_{\xi_{m}^-} 
& &M^{m+1}(v) \ar[ld]^{\xi_{m}^+} \ar[rd] &\cdots \\
& & &M^{m, m+1}(v) & & 
}
\end{equation}
connecting the moduli space on $\bP^2$ 
and that on the blow-up $\hatP^2$. 
The intermediate models $M^m(v)$ 
also have modular interpretations; 
they are the moduli spaces of {\it $m$-stable sheaves} 
(see Definition \ref{defin:m-st}). 
Similarly, the space $M^{m, m+1}(v)$ parametrizes 
$m$-stable and $(m+1)$-stable sheaves 
with various Chern characters. 

In these examples, 
the moduli spaces appearing in wall crossing diagrams 
are smooth and birational to each other. 
In fact, in the case studied in \cite{eg95, fq95}, 
the moduli spaces are connected by standard flips. 
In the case of \cite{ny11a, ny11b, ny11c}, 
their geometry is more complicated. 
Indeed, Nakajima-Yoshioka proved that 
the contracted loci of the morphisms 
$\xi_{m}^{\pm}$ have stratifications 
(called Brill-Noether stratifications) 
such that each stratum has the structure 
of a Grassmannian bundle. 

The aim of this paper is the further study of 
birational geometric properties of the diagram (\ref{eq:nyintro}). 
In particular, we show that it 
is an instance of the minimal model program. 
\begin{thm}[Theorem \ref{mmp}] \label{mmpintro}
The diagram (\ref{eq:nyintro}) realizes a minimal model program 
for the moduli space of framed torsion free sheaves 
on the blow-up $\hatP^2$. 
The program ends with the minimal model, 
the moduli space of framed torsion free sheaves on $\bP^2$, 
which is a hyper-K$\ddot{a}$hler manifold. 
\end{thm}

We will also verify Bondal-Orlov \cite{bo95}, 
Kawamata's \cite{kaw02} D/K equivalence conjecture 
for these moduli spaces: 

\begin{thm}[Theorem \ref{derived}]
For each integer $m \in \mathbb{Z}_{\geq 0}$, 
we have a fully faithful embedding 
\[
D^b(M^m(v)) \hookrightarrow D^b(M^{m+1}(v))
\]
between the derived categories. 
In particular, we have an embedding 
\[
D^b(M_{\bP^2}) \hookrightarrow D^b(M_{\hatP^2}), 
\]
where $M_{\bP^2}, M_{\hatP^2}$ 
denote the moduli spaces of torsion free framed sheaves 
on $\bP^2, \hatP^2$, respectively. 
\end{thm}

So we get an interesting relationship 
among wall-crossing formulas for Donaldson type invariants, 
birational geometry, and derived categories. 

We can also consider the moduli space 
of Gieseker stable sheaves 
on a smooth projective surface 
and that on its blow-up 
(see Theorem \ref{thm:proj} for the precise statement): 

\begin{thm}
Let $S$ be a smooth projective surface, 
$\hat{S}$ the blow-up of $S$ at a point. 
Under certain numerical conditions, 
the MMP for the moduli space $M_{\hat{S}}$ 
of Gieseker stable sheaves on $\hat{S}$ 
is reduced to MMP for the moduli $M_{S}$ on $S$. 
Furthermore, there exists a fully faithful embedding 
\[
D^b(M_{S})
\hookrightarrow 
D^b(M_{\hat{S}}) 
\] 
between their derived categories. 
\end{thm}
For instance, we can apply the above theorem 
when $S$ is a del Pezzo surface. 

\subsection{Strategy of the proof}
To prove our main result Theorem \ref{mmpintro}, 
we will compute the normal bundles of the fibers explicitly, 
following the idea from \cite{eg95, fq95}. 
Although the geometry of the diagram (\ref{eq:nyintro}) 
is more complicated compared to the one considered in \cite{eg95, fq95}, 
it turns out their method still works in our setting. 
Actually, we are able to describe 
the normal bundle of each Brill-Noether stratum explicitly.  
Then we will see that the diagram realizes the  MMP 
when we decrease the stability parameter $m \in \mathbb{Z}_{\geq 0}$. 

Furthermore, the normal bundle computation enables us 
to reduce the construction of the fully faithful embedding 
between derived categories to the formal local case; 
the latter is already handled in the paper \cite{blvdb16} 
and hence we can prove Theorem \ref{derived}.

\subsection{Relation to existing works}
In \cite{kos17}, 
the author studied birational geometry of 
the Hilbert scheme of two points on blow-ups. 
The main result of the present paper 
is an extension to the completely general setting. 

There are several works investigating 
birational geometry and derived categories of moduli spaces. 
For a standard flip between moduli spaces 
obtained in \cite{eg95, fq95}, Ballard \cite{bal17} constructed 
a semi-orthogonal decomposition (SOD) of their derived categories. 

Recently, Toda \cite{Toddbir, Todsemi} 
introduced the notion of d-critical birational geometry, 
which is a certain virtual analogue 
of usual birational geometry. 
It is shown that if two smooth varieties 
are connected by a simple d-critical flip, 
then we have an SOD 
of their derived categories. 
See \cite{kt19, Todsemi} for interesting examples of d-critical flips. 

The SODs obtained in the papers \cite{bal17}, \cite{Todsemi} 
can be considered as categorifications of 
wall crossing formulas for Donaldson type invariants, 
Donaldson-Thomas type invariants, respectively. 
It would be interesting to describe 
the semi-orthogonal complements of the embedding 
in our Theorem \ref{derived}, which would give 
a categorification of Nakajima-Yoshioka's wall crossing formula.

\subsection{Organization of the paper}
The paper is organized as follows. 
In Section \ref{sec:prel}, 
we collect some terminology and useful lemmas 
from birational geometry and derived categories. 
In Section \ref{sec:frame}, we recall 
the result of Nakajima-Yoshioka. 
In Section \ref{sec:bir}, we prove our main results. 
In Section \ref{sec:ex}, we give some explicit examples. 

\begin{ACK}
I would like to thank Professors Jim Bryan 
and Yukinobu Toda for fruitful discussions. 
This work was supported by Grant-in-Aid for JSPS Research Fellow 17J00664. 

Finally, I would like to thank the referee 
for various suggestions and comments to the previous version of this paper. 
\end{ACK}

\begin{NaC}

In this paper, we always work over 
the complex number field $\mathbb{C}$. 
\begin{itemize}
\item For a variety $X$, 
we denote by 
$D^b(X):=D^b(\Coh(X))$ 
the bounded derived category of 
coherent sheaves on $X$. 

\item For a proper morphism 
$f \colon M \to N$ 
between varieties 
and objects $E, F \in D^b(M)$, 
we denote by 
$\mathcal{E}xt^q_{f}(E, F)$ 
the $q$-th derived functor of 
$\mathcal{H}om_{f}(E, F):=
f_{*}\mathcal{H}om(E, F)$. 

\item 
For coherent sheaves $E, F$ 
on a variety, we define 
$\hom(E, F):=\dim\Hom(E, F)$ 
and 
$\ext^i(E, F):=\dim\Ext^i(E, F)$. 

\item For a vector bundle $\mcV$ 
on a variety and an integer $i >0$, 
we denote by 
$\Gr(i, \mcV)$ the Grassmann bundle 
of $i$-dimensional {\it subbundles} of $\mcV$. 
\end{itemize}
\end{NaC}

\section{Preliminaries} \label{sec:prel}
\subsection{Terminologies from birational geometry}
In this subsection, 
we recall some notions from birational geometry. 
The standard reference for this subsection is \cite{km98}. 

\begin{defin} \label{def:k-neg/pos}
Let $\phi \colon X \to Z$ be a projective morphism between 
normal quasi-projective varieties. 
We say that $\phi$ is 
a {\it $K$-positive (resp. $K$-negative) contraction} 
if the following conditions hold: 
\begin{enumerate}
\item $\phi_*\mcO_X \cong \mcO_Z$, 
\item the canonical divisor $K_X$ 
(resp. the anti-canonical divisor $-K_X$) 
is $\phi$-ample. 
\end{enumerate}
\end{defin}

\begin{defin} \label{def:contr}
Let $\phi \colon X \to Z$ be a $K$-negative contraction. 
\begin{enumerate}
\item $\phi$ is called a {\it divisorial (resp. flipping) contraction} 
if it is birational and the $\phi$-exceptional locus 
has codimension one (resp. at least two). 
\item $\phi$ is called a {\it Mori fiber space} 
if we have $\dim X > \dim Z$. 
\item Assume that $\phi$ is a flipping contraction. 
Then a {\it flip} of $\phi$ is a $K$-positive birational contraction 
$\phi^+ \colon X^+ \to Z$. 
We also call the birational map $X \dashrightarrow X^+$ 
a flip. 
\end{enumerate}
\end{defin}

\begin{defin} \label{def:mmp}
Let $X$ be a quasi-projective variety with at worst terminal singularities. 
A {\it minimal model program} of $X$ is a sequence of birational maps 
\[
X=X_0 \dashrightarrow X_1 \dashrightarrow \cdots \dashrightarrow X_N,
\]
such that 
\begin{enumerate}
\item each birational map $X_i \dashrightarrow X_{i+1}$ 
is either a divisorial contraction or a flip, 
\item the variety $X_N$ is either a minimal model (i.e. $K_{X_N}$ is nef) 
or has a structure of a Mori fiber space. 
\end{enumerate}
\end{defin}

We do not give the definition of a terminal singularity, 
as we only consider smooth varieties in this paper. 
For the precise definition, see \cite[Definition 2.34]{km98}. 

The following lemma is useful for our purpose: 
\begin{lem}[{\cite[Proposition 3.7]{aw97}}] \label{lem:formal}
Let $X$ be a smooth variety, 
$\phi \colon X \to Z$ a $K$-negative contraction, 
and $F \subset X$ a smooth $\phi$-fiber. 
Assume that the following conditions hold: 
\begin{itemize}
\item the conormal bundle 
$\mcN^{\vee}:=\mcN^{\vee}_{F/X}$ is nef, 

\item $H^1(F, \mcT_{F} \otimes \Sym^i(\mcN^{\vee}))
=H^1(F, \mcN \otimes \Sym^i(\mcN^{\vee}))=0$ 
for $i \geq 1$. 
\end{itemize}

Then the formal neighborhood of $F$ in $X$ 
is isomorphic to that of $F$ in the total space of $\mcN$, 
embedded as the zero section. 
\end{lem}

\subsection{Fourier-Mukai functors}
In this subsection, 
we recall the definition and basic properties of Fourier-Mukai functors. 
The standard reference is \cite{huy06}. 

\begin{defin}
Let $X, Y$ be smooth quasi-projective varieties, 
and $\mcP \in D^b(X \times Y)$ be an object 
whose support is proper over $Y$. 
The {\it Fourier-Mukai (FM) functor with kernel $\mcP$} 
is the functor $\Phi^\mcP \colon D^b(X) \to D^b(Y)$ 
defined by 
\[
D^b(X) \ni E \mapsto \Phi^\mcP(E):=
\dR p_{Y*} \left(p_X^*E \otimes^\dL \mcP \right) \in D^b(Y), 
\]
where $p_X \colon X \times Y \to X$, 
$p_Y \colon X \times Y \to Y$ 
denote the projections. 
\end{defin}
Note that, in the above definition, 
we assume the object $\mcP$ has 
proper support over $Y$ to ensure 
the associated FM functor $\Phi^\mcP$ 
preserves bounded complexes. 

Let us give a trivial example: 
\begin{ex}
Let $\mcO_{\Delta_X}$ be the structure sheaf 
of the diagonal $\Delta_X \subset X \times X$. 
Then the associated FM functor 
$\Phi^{\mcO_{\Delta_X}}$ 
coincides with the identity functor 
$\id_{D^b(X)}$. 
\end{ex}

Let $\Phi=\Phi^\mcP \colon D^b(X) \to D^b(Y)$ 
be a FM functor with kernel $\mcP$. 
By \cite[Proposition 5.9]{huy06}, 
the right adjoint functor $\Phi^\dR \colon D^b(Y) \to D^b(X)$ 
is given by the FM functor with kernel 
\[
\mcP_\dR := \mcP^\vee \otimes p_X^* \omega_X[\dim X]. 
\]

Similarly, the composition of two FM functors 
is again a FM functor: 
let $\Phi^\mcQ \colon D^b(Y) \to D^b(Z)$ be another FM functor. 
Then the composition 
$\Phi^\mcQ \circ \Phi^\mcP \colon D^b(X) \to D^b(Z)$ 
is a FM functor with kernel 
\[
\mcP * \mcQ :=\dR p_{XZ*} \left(p_{XY}^*\mcP \otimes^\dL p_{YZ}^* \mcQ \right) 
\in D^b(X \times Z), 
\]
where 
$p_{XY} \colon X \times Y \times Z \to X \times Y$, 
$p_{YZ} \colon X \times Y \times Z \to Y \times Z$, 
$p_{XZ} \colon X \times Y \times Z \to X \times Z$ 
denote the projections, 
see \cite[Proposition 5.10]{huy06}. 

Now let us consider the adjoint map 
\begin{equation} \label{eq:adjmap}
\Phi^{\mcO_{\Delta_X}}=\id_{D^b(X)} \to 
\Phi^{\mcP_\dR} \circ \Phi^\mcP=\Phi^{\mcP_R * \mcP}. 
\end{equation}
The following result tells us that 
it lifts to the morphism between FM kernels: 

\begin{prop}[{\cite[Theorem 3.2, Proposition 3.6]{al12}}]
Let $\Phi=\Phi^\mcP \colon D^b(X) \to D^b(Y)$ be a FM functor. 
Then there exists a morphism 
\begin{equation}\label{eq:adjobj}
\mcO_{\Delta_X} \to \mcP_\dR * \mcP 
\end{equation}
in $D^b(X \times X)$, 
which induces the adjoint map (\ref{eq:adjmap}). 
\end{prop}

We will use the following lemma 
in the proof of the main result: 
\begin{lem} \label{lem:formal}
Let $X, Y$ be smooth quasi-projective varieties, 
$X \to U$, $Y \to U$ projective morphisms to a variety $U$. 
Take an object $\mcP \in D^b(X \times Y)$ 
supported on the fiber product $X \times_U Y$. 
Denote by  $\mcQ \in D^b(X \times X)$ 
the cone of the morphism 
$\mcO_{\Delta_X} \to \mcP_\dR * \mcP$ 
in (\ref{eq:adjobj}). 
The following statements hold: 
\begin{enumerate}
\item The FM functor $\Phi=\Phi^\mcP$ 
is fully faithful if and only if $\mcQ=0$. 
\item For a point $p \in U$, denote by $\hat{U}_p$ 
the completion of $U$ at $p$. 
Denote by $\hat{X}_p$ 
the base change of $X \to U$ 
along $\hat{U}_p \to U$. 
Then $\mcQ=0$ if and only if 
$\mcQ \otimes_{\mcO_{X \times X}} \mcO_{\hat{X}_p \times \hat{X}_p}=0$ 
for every $p \in U$. 
\end{enumerate}

In particular, the FM functor $\Phi^\mcP$ is fully faithful 
if and only if the FM functors $\Phi^{\mcP^\wedge_p}$ 
are fully faithful for all $p \in U$, where we put 
$\mcP^\wedge_p :=
\mcP \otimes_{\mcO_{X \times Y}} \mcO_{\hat{X}_p \times \hat{Y}_p}$. 
\end{lem}
\begin{proof}
(1) First observe that the functor $\Phi$ is fully faithful 
if and only if the adjoint map (\ref{eq:adjmap}) is an isomorphism. 
Let us take an object $E \in D^b(X)$. 
By applying $\Phi^{(-)}(E)$ to the exact triangle 
\[
\mcO_{\Delta_X} \to \mcP_\dR * \mcP \to \mcQ, 
\]
in $D^b(X \times X)$, 
we obtain an exact triangle 
\begin{equation} \label{eq:functortriangle}
E \to \Phi^\dR \circ \Phi(E) \to \Phi^\mcQ(E)
\end{equation}
in $D^b(X)$. 
Hence if $\mcQ=0$, we have an isomorphism 
$E \xrightarrow{\cong} \Phi^\dR \circ \Phi(E)$ for all $E \in D^b(X)$, 
i.e., $\Phi$ is fully faithful. 

For the converse, 
assume that $\mcQ \neq 0$. 
Then there exists a point $x \in X$ such that 
$\mcQ_x:=\mcQ|_{\{x\} \times X} \neq 0$. 
In particular, we have 
$\Phi^\mcQ(\mcO_x)=\mcQ_x \neq 0$ 
and hence 
$\mcO_x \neq \Phi^\dR \circ \Phi(\mcO_x)$ 
by the exact triangle (\ref{eq:functortriangle}), 
i.e., the functor $\Phi$ is not fully faithful. 

(2) The second assertion follows from the fact that 
the completion $\hat{U}_p \to U$ is faithfully flat. 
\end{proof}

\subsection{Grassmannian flip} \label{subsec:gflip}
Here we recall about 
geometry and the derived categories of {\it Grassmannian flips}, 
which play important roles in this paper. 
We refer \cite[Chapter II]{acgh85} for the details 
(see also \cite[Section 1]{blvdb16}). 
Let $W^\pm$ be vector spaces. 
Take a positive integer $i \leq \min\{\dim W^\pm \}$. 

Then the {\it determinantal variety} is defined to be 
\[
Z=Z_i:= \left\{
a \in \Hom\left((W^-)^\vee, W^+ \right): 
\rk(a) \leq i 
\right\}, 
\]
and its {\it Springer resolution} is defined as 
\begin{align*}
Y^+=Y^+_i
&:=\left\{
(a, V) \in \Hom\left((W^-)^\vee, W^+ \right) \times \Gr(i, W^+): 
\Image(a) \subset V
\right\}. 
\end{align*}
We have the following projections: 
\[
\xymatrix{
&Y^+ \ar[r]^-{h^+} \ar[d]_{\phi^+}& \Gr(i, W^+) \\
&Z. 
}
\]

The fiber of 
$h^+ \colon Y^+ \to \Gr(i, W^+)$ 
over a point 
$V \in \Gr(i, W^+)$ is 
\[
\Hom\left((W^-)^\vee, V \right) 
\cong V \otimes W^-. 
\]
Hence we have an isomorphism 
\[
Y^+ \cong \Tot_{\Gr(i, W^+)}\left( 
S^+ \otimes W^-
\right), 
\]
where $S^+$ denotes the universal subbundle 
on $\Gr(i, W^+)$. 
This shows that the variety $Y^+$ 
is smooth of dimension 
$i(\dim W^+ + \dim W^- -i)$. 
Moreover, the other projection 
$\phi^+ \colon Y^+ \to Z$ is isomorphism 
over the open locus 
\[
Z^o:=\{a \in Z: \rk(a)=i \}, 
\]
therefore $Y^+ \to Z$ 
is actually a resolution of singularities. 

On the other hand, 
we have canonical isomorphisms 
\[
\Hom\left((W^+)^\vee, W^- \right) 
\cong W^- \otimes W^+ 
\cong \Hom\left((W^-)^\vee, W^+ \right), 
\]
hence by replacing the roles of $W^\pm$ 
in the above construction, 
we get a second resolution of singularities of the variety $Z$: 
\[
\xymatrix{
&Y^- \cong \Tot_{\Gr(i, W^+)}\left( S^+ \otimes W^-\right) \ar[r]^-{h^-} \ar[d]_{\phi^-}& \Gr(i, W^-) \\
&Z, 
}
\]
where $S^-$ denotes the universal subbundle of $\Gr(i, W^-)$. 

Hence we obtain the diagram 
\[
\xymatrix{
&Y^- \ar[rd] & &Y^+ \ar[ld] \\
& &Z, &
}
\]
which we call the {\it Grassmannian flip}. 
The following lemma justifies this notion: 

\begin{lem} \label{lem:KofY}
Assume that $\dim W^+ > \dim W^-.$ 
Then the following statements hold: 
\begin{enumerate}
\item The canonical bundles of $Y^\pm$ are given as 
\[
\omega_{Y^\pm}
=h^{\pm *}\left(\det(S^\pm)^{\otimes \pm (\dim W^+ - \dim W^-)} \right). 
\]
In particular, the morphism $\phi^+$ (resp. $\phi^-$) 
is a $K$-negative (resp. $K$-positive) contraction. 

\item When $i=\dim W^-$, the morphism $\phi^+$ is a divisorial contraction 
and the morphism $\phi^-$ is an isomorphism. 

\item When $i < \dim W^-$, 
the birational map $Y^+ \dashrightarrow Y^-$ is a flip. 
\end{enumerate}
\end{lem}
\begin{proof}
(1) We have the following exact sequence 
\[
0 \to h^{\pm *}\left(S^\pm \otimes W^\mp \right) 
\to \mcT_{Y^\pm} \to h^{\pm *} \mcT_{\Gr(i, W^\pm)} \to 0, 
\]
and the tangent bundle of the Grassmannian variety is given by 
$\mcT_{\Gr(i, W^\pm)} \cong S^{\pm *} \otimes Q^\pm$, 
where $Q^\pm$ denotes the universal quotient bundle 
of $\Gr(i, W^\pm)$. 
Hence we obtain 
\begin{equation} \label{eq:KofY}
\begin{aligned}
\omega_{Y^\pm} 
&\cong 
h^{\pm *}\left( \det\left(S^\pm \right)^{\otimes - \dim W^\mp} 
\otimes \left(\det(S^\pm)^{\otimes (\dim W^\pm -i)} \otimes \det(Q^{\pm *})^{\otimes i} \right) 
\right) \\
&\cong h^{\pm *} \left(
\det(S^\pm)^{\otimes (\dim W^\pm - \dim W^\mp)} 
\otimes \left(
\det(S^\pm) \otimes \det(Q^\pm)
\right)^{\otimes -i}
\right). 
\end{aligned}
\end{equation}

Moreover, the tautological sequence 
\[
0 \to S^\pm \to W^\pm \otimes \mcO_{\Gr(i, W^\pm)} \to Q^\pm \to 0
\]
shows that 
\[
\det(S^\pm) \otimes \det(Q^\pm) \cong \mcO_{\Gr(i, W^\pm)}. 
\]
Now the equation (\ref{eq:KofY}) becomes 
\[
\omega_{Y^\pm} \cong 
h^{\pm *} \left(
\det(S^\pm)^{\otimes (\dim W^\pm - \dim W^\mp)} \right)
\]
as required. 

To prove (2) and (3), we need to determine 
the dimensions of the $\phi^\pm$-exceptional loci. 
First note that we have a sequence of closed immersions
\[
\{0\}=Z_0 \subset Z_1 \subset \cdots Z_{i-1} \subset Z_i=Z, 
\]
and we have 
$\dim (\phi^\pm)^{-1}(Z_{j}) \leq 
\dim (\phi^\pm)^{-1}(Z_{i-1})$ 
for all $j \leq i-1$. 

Let us take a point 
$a \in Z_{i-1} \setminus Z_{i-2}$. 
By definition, the subspace 
$V_a:=\Image(a) \subset W^\pm$ has dimension $i-1$. 
The fibers $Y^\pm_a:=(\phi^\pm)^{-1}(a)$ are
\[
Y^\pm_a \cong 
\left\{
V \in \Gr(i, W^\pm): V_a \subset V
\right\} 
\cong \Gr(1, W^\pm/V_a). 
\]
Hence we have 
\begin{align*}
\dim \Exc(\phi^\pm)
&=\dim Z_{i-1}+\dim \Gr(1, W^\pm/V_a) \\
&=\dim Y^\pm_{i-1}+\dim \Gr(1, W^\pm/V_a) \\
&=(i-1)(\dim W^\pm+\dim W^\mp-i+1)
	+(\dim W^\pm-(i-1)-1) \\
&=i(\dim W^\pm + \dim W^\mp-i)+i-\dim W^\mp-1 \\
&=\dim Y^\pm -(\dim W^\mp+1-i), 
\end{align*}
and the assertions (2) and (3) hold. 
\end{proof}

For the derived categories of $Y^\pm$, 
we have the following result: 
\begin{thm}[{\cite[Theorem D]{blvdb16}}] \label{thm:blvdb}
Assume that $\dim W^+ > \dim W^-$. 
Then the FM functor 
\[
\Phi^{\mcO_{W}} \colon 
D^b(Y^-) \hookrightarrow D^b(Y^+)
\]
is fully faithful, 
where we put $W:= Y^- \times_Z Y^+$. 
\end{thm}

\begin{rmk}
When $\dim W^-=\dim W^+$, 
the canonical bundles of $Y^\pm$ are trivial, 
and we call the birational map 
$Y^+ \dashrightarrow Y^-$ 
the Grassmannian {\it flop}. 
In this case, the FM functor $\Phi^{\mcO_W}$ is an equivalence, 
which is treated also in \cite{ds14}. 
\end{rmk}

\subsection{Bott type vanishing}
For later use, we prove some vanishing results on 
cohomology groups of certain vector bundles on the Grassmannian varieties. 
Let $W$ be a vector space of dimension $n$, 
and $i \in \bZ_{>0}$ be a positive integer with $i < n$. 
We consider the Grassmannian variety $G=\Gr(i, W)$. 
Denote by $S, Q$ the universal sub and quotient bundles, respectively. 

An element 
$\beta=(\beta_1, \cdots, \beta_s) \in \bZ^{\oplus s}$ 
is called a {\it weight} if it satisfies 
$\beta_1 \geq \beta_2 \geq \cdots \geq \beta_{s-1} \geq \beta_s$. 
For a given pair $\alpha=(\beta, \gamma)$ of 
weights $\beta \in \bZ^{\oplus i}$ and $\gamma \in \bZ^{\oplus n-i}$, 
we define the vector bundle $V(\alpha)$ as 
\[
V(\alpha):=K_\beta(S^\vee) \oplus K_\gamma(Q^\vee), 
\]
where $K_\beta(-), K_\gamma(-)$ 
denote the {\it Weyl functors} 
(cf. \cite[Section 2.1]{wey03}). 

\begin{ex}
\begin{enumerate}
\item For $\alpha=(k, 0, \cdots, 0) \in \bZ^{\oplus n}$ 
with $k \geq 0$, we have 
$V(\alpha)=\Sym^k(S^\vee)$. 

\item For $\alpha=(0, \cdots, 0, -1) \in \bZ^{\oplus n}$, 
we have 
$V(\alpha)=Q$. 
\end{enumerate}
\end{ex}

We will use the following version of Bott vanishing: 
\begin{thm} \label{thm:Bott}
Assume that an element 
$\alpha=(\beta_1, \cdots, \beta_i, \gamma_{i+1}, \cdots, \gamma_n) \in \bZ^{\oplus n}$ 
satisfies 
$\beta_1 \geq \cdots \beta_i \geq 
\gamma_{i+1} \geq \cdots \geq \gamma_n$. 
Then we have the vanishing 
\[
H^j\left(G, V(\alpha) \right)=0 
\]
for all $j > 0$. 
\end{thm}
\begin{proof}
This is a special case of \cite[Corollary 4.1.9 (2)]{wey03}. 
\end{proof}

As an application of the above Bott type vanishing, 
we have the following: 
\begin{lem} \label{lem:vanGrS}
We have the following vanishing of cohomology groups: 
\begin{enumerate}
\item $H^1\left(G, \mcT_G \otimes \Sym^k(S^\vee) \right)=0$ for all $k \geq 0$. 
\item $H^1\left(G, S \otimes \Sym^k(S^\vee) \right)=0$ for all $k >0$. 
\item $H^1\left(G, \Sym^k(S^\vee) \right)=0$ for all $k >0$. 
\end{enumerate}
\end{lem}
\begin{proof}
We only prove the first assertion, 
since the other can be proved 
by similar computations. 
Recall that we have 
$\mcT_G \cong S^\vee \otimes Q$, 
and hence 
\[
\mcT_G \otimes \Sym^k(S^\vee) 
\cong S^\vee \otimes \Sym^k(S^\vee) \otimes Q. 
\]

Moreover, by Pieri formula (cf. \cite[Corollary 2.3.5]{wey03}), 
we have 
\[
S^\vee \otimes \Sym^k(S^\vee) 
\cong K_{\nu_1}(S^\vee) \oplus K_{\nu_2}(S^\vee), 
\]
where 
$\nu_1=(k, 1, 0, \cdots, 0), 
\nu_2=(k+1, 0, \cdots, 0) \in \bZ^{\oplus i}$ 
(see also \cite[Proposition 2.1.18]{wey03} for the relation between 
the Schur and the Weyl functors). 
Hence we have 
\[
\mcT_G \otimes \Sym^k(S^\vee) 
\cong V(\alpha_1) \oplus V(\alpha_2), 
\]
where we put 
$\alpha_t=(\nu_t, 0, \cdots, 0, -1) \in \bZ^{\oplus n}$ 
for $t=0, 1$. 
Hence by Theorem \ref{thm:Bott}, 
we obtain the vanishing 
$H^1\left(G, \mcT_G \otimes \Sym^k(S^\vee) \right)=0$ 
as required. 
\end{proof}

\section{Framed sheaves on the blow-up} \label{sec:frame}
\subsection{ADHM description and wall crossing}
Let us consider the projective plane 
$\bP^2=\bC^2 \cup l_{\infty}$ 
and its blow-up $f \colon \hatP^2 \to \bP^2$ 
at the origin $0 \in \bC^2$. 
Denote by $C \subset \bP^2$ 
the $f$-exceptional curve. 
In this subsection, 
we recall the notion of $m$-stable sheaves on $\hatP^2$ 
with framing at $l_{\infty}$ 
studied by Nakajima-Yoshioka \cite{ny11a}. 

\begin{defin}\label{defin:m-st}
Fix a non-negative integer $m \in \mathbb{Z}_{\geq 0}$. 
Let $(E, \Phi)$ be a framed sheaf on $\hatP^2$, i.e., 
let $E$ be a coherent sheaf on $\hatP^2$ and 
$\Phi \colon E|_{l_{\infty}} \xrightarrow{\cong} 
\mathcal{O}_{l_{\infty}}^{\oplus \ch_0(E)}$ 
a framing at $l_{\infty}$. 
We say that $(E, \Phi)$ is 
{\it $m$-stable} if the following conditions hold: 
\begin{enumerate}
\item $\Hom(E(-mC), \mathcal{O}_{C}(-1))=0$, 
\item $\Hom(\mathcal{O}_{C}, E(-mC))=0$, 
\item $E$ is torsion free outside $C$.  
\end{enumerate}
\end{defin}

\begin{rmk}
Let $E$ be an $m$-stable framed sheaf. 
By \cite[Proposition 1.9 (1)]{ny11b}, 
the condition (1) in the above definition implies that 
the sheaf $E(-mC)$ is an object of 
the category $\Per(\hatP^2/\bP^2)$ 
of {\it perverse coherent sheaves} 
introduced by Bridgeland \cite{bri02}. 
Hence we may also think of $m$-stable sheaves 
as stable objects in the category 
$\Per(\hatP^2/\bP^2) \otimes \mcO(mC)$. 
\end{rmk}

We denote by $M^m(v)$ the fine moduli space 
of $m$-stable framed sheaves on $\hatP^2$ 
with Chern character $v$. 
Let us recall the ADHM description of 
framed sheaves on $\hatP^2$; 
Consider the following quiver $Q$ 
\[
\xymatrix{
&0 \ar@<-0.8ex>[rr]_-d & 
&1 \ar@<-1.0ex>[ll]_{B_{1}, B_{2}} \ar@<-0.1ex>[ll] \ar@<+0.2ex>[ld]^-j \\
& &\infty \ar@<+0.2ex>[lu]^-i & 
}
\]
with the relation 
\[
I \colon B_{1}dB_{2}-B_{2}dB_{1}+ij=0. 
\]

For a given Chern character 
\[
v=(r, -kC, \ch_{2}) \in \bigoplus_{i=0}^2 H^{2i}(\hatP^2, \mathbb{Q}), 
\]
we associate the dimension vector 
$\vec{d}=(d_{0}, d_{1}, d_{\infty})$ 
by the following formula: 
\begin{equation} \label{eq:dim}
d_{\infty}=r, \quad k=d_{0}-d_{1}, 
\quad \ch_2=-\frac{1}{2}(d_{0}+d_{1}). 
\end{equation}

\begin{thm}[\cite{ny11a}] \label{thm:adhm}
Let us fix a Chern character 
$v=(r, -kC, \ch_{2}) \in H^{2*}(\hatP^2, \mathbb{Q})$
with $r>0, k\geq 0$, 
and associate the vector 
$\vec{d} \in \mathbb{Z}_{\geq 0}^{\oplus 3}$ 
as in (\ref{eq:dim}). 
Then, in the region 
$\Omega:=\left\{
(\zeta_{0}, \zeta_{1}) \in \mathbb{R}^2 \colon 
\zeta_{0} > 0
\right\}$ 
of King's stability parameters for $(Q, I)$-representations, 
the walls are classified as 
\[
\mathcal{W}_{m}:=\left\{
(\zeta_{0}, \zeta_{1}) \in \Omega \colon 
m\zeta_{0}+(m+1)\zeta_{1}=0
\right\}, 
\quad m \in \mathbb{Z}_{\geq 0}. 
\]

For $m \geq 1$, denote by $\mathcal{C}_{m}$ the chamber 
between the walls $\mathcal{W}_{m}$ and $\mathcal{W}_{m-1}$, 
and put 
$\mathcal{C}_{0}:=
\left\{
(\zeta_{0}, \zeta_{1}) \in \mathbb{R}^2 : 
\zeta_{0}, \zeta_{1} <0
\right\}. 
$
The following statements hold: 
\begin{enumerate}
\item For each integer $m \in \mathbb{Z}_{\geq 0}$, 
and a stability condition $\zeta \in \mathcal{C}_{m}$, 
there exists an isomorphism 
$M^m(v) \cong M^\zeta(Q, I; \vec{d})$. 
Here, $M^\zeta(Q, I; \vec{d})$ denotes 
the moduli space of $\zeta$-stable (Q, I)-representations 
with dimension vector $\vec{d}$. 
Furthermore, these moduli spaces are 
either empty or 
smooth quasi-projective varieties 
of dimension 
$d_{\infty}(d_{0}+d_{1})-(d_{0}-d_{1})^2$. 
\item We have a natural morphism 
\[
M^0(v) \to M_{\bP^2}(r, 0, \ch_2), 
\quad E \mapsto f_{*}E, 
\]
which is an isomorphism when $k=0$. 
Here, $M_{\bP^2}(r, 0, \ch_2)$ 
denotes the moduli space of 
torsion free framed sheaves on $\bP^2$. 

\item There exists an integer 
$m_{0} \in \mathbb{Z}_{\geq 0}$ such that 
for every integer $m \geq m_{0}$, 
we have an isomorphism 
\[
M^m(v) \cong M_{\hatP^2}(v), 
\] 
where $M_{\hatP^2}(v)$ 
denotes the moduli space of 
torsion free framed sheaves on $\hatP^2$.
 
\item For each integer $m \in \mathbb{Z}_{\geq 0}$ 
and a stability condition $\zeta \in \mathcal{W}_{m}$, 
there exists a set-theoretic bijection 
\[
M^{\zeta}(Q, I; \vec{d})
=M^{m, m+1}(v):=\bigsqcup_{i \geq 0} 
               \left(M^m(v-ic_{m}) \cap M^{m+1}(v-ic_{m}) 
               \right), 
\]
where we put $c_{m}:=\ch(\mathcal{O}_{C}(-m-1))$. 
\end{enumerate}
\end{thm}
\begin{proof}
The classification of walls is explained in \cite[Section 4.3]{ny11a}; 
The statements (1), (2), (3) are 
proved in Theorems 1.5 and 2.5, 
Proposition 7.4, 
and Proposition 7.1 
in \cite{ny11a}, respectively. 

Let us now consider the statement (4). 
By \cite[Section 4.3, Proposition 5.3]{ny11a}, 
the wall $\mcW_m$ corresponds to 
the destabilizing object $\mcO_C(-m-1)$ 
in the following sense: 
for any stability condition $\zeta \in \mcW_m$ on the wall 
and any $\zeta$-semistable object $E$, its S-equivalence class 
is of the form 
$\mcO_C(-m-1)^{\oplus i} \oplus E'$ 
for some non-negative integer $i \geq 0$ 
and $m$-stable and $(m+1)$-stable sheaf $E'$ 
with $\ch(E')=v-ic_m$. 
Hence as a closed point in the moduli space, we have 
\[
[E]=[\mcO_C(-m-1)^{\oplus i} \oplus E'] \in M^\zeta(Q, I; \vec{d}), 
\]
i.e., the closed point $[E] \in M^\zeta(Q, I; \vec{d})$ 
is uniquely determined by the sheaf 
$E' \in M^m(v-ic_m) \cap M^{m+1}(v-ic_m)$. 
Hence we have the bijection as stated. 
\end{proof}

By the above theorem, we have the diagram 
as in (\ref{eq:nyintro})
connecting the moduli spaces 
$M_{\bP^2}(r, 0, \ch_{2})$ 
and $M_{\hatP^2}(r, -kC, \ch_{2})$.

\subsection{Brill-Noether loci} \label{subsec:bn}

Next we recall the Brill-Noether stratifications on the moduli spaces 
and the determination of the fibers over $\xi_{m}^{\pm}$. 
For each integer $i \in \mathbb{Z}_{\geq 0}$, 
let us consider the following locally closed subschemes: 
\begin{align*}
&M^m(v)_{i}:=\left\{
(E, \Phi) \in M^m(v) \colon 
\hom(\mathcal{O}_{C}(-m-1), E)=i
\right\}, \\
&M^{m+1}(v)_{i}:=\left\{
(E, \Phi) \in M^{m+1}(v) \colon 
\hom(E, \mathcal{O}_{C}(-m-1))=i
\right\}. 
\end{align*}
We call them as the {\it Brill-Noether strata}. 
We also denote as 
$M^{m, m+1}(v)_{i}:=M^m(v-ic_{m}) \cap M^{m+1}(v-ic_{m})$. 

Let us take an object $E \in M^m(v)$. 
By \cite[Proposition 3.15]{ny11b}, 
we have the exact sequence 
\begin{equation} \label{eq:destabE-}
0 \to \Hom(\mcO_C(-m-1), E) \otimes \mcO_C(-m-1) \xrightarrow{ev} 
E \to E' \to 0
\end{equation}
for some object $E' \in M^{m, m+1}(v)$. 
By (the proof of) Theorem \ref{thm:adhm} (4), we have 
$\xi^-_m(E)=E'$, 
where $\xi^-_m \colon M^m(v) \to M^{m, m+1}(v)$ 
is the morphism in the diagram (\ref{eq:nyintro}). 
Hence the morphism $\xi^-_m$ restricts to the morphism 
$\xi^-_{m, i} \colon M^m(v)_i \to M^{m, m+1}(v)_i$. 
Moreover, from the exact sequence (\ref{eq:destabE-}), 
we see that an $m$-stable sheaf $E$ is $(m+1)$-stable 
if and only if the vanishing 
$\Hom(\mcO_C(-m-1), E)=0$ holds. 
Hence we have 
\[
M^m(v)_0=M^m(v) \cap M^{m+1}(v) \subset M^m(v) 
\]
and it is an open immersion. 

Let us denote by 
\[
\mcE'_{i} \in \Coh\left(\hatP^2 \times M^{m, m+1}(v)_{i} \right) 
\]
the universal family, and let 
\[
p \colon \hatP^2 \times M^{m, m+1}(v)_{i} \to M^{m, m+1}(v)_{i}, \quad 
q \colon \hatP^2 \times M^{m, m+1}(v)_{i} \to \hatP^2 
\]
be the projections. 
The following theorem shows 
the structure of the morphism $\xi^\pm_m$ 
in terms of the Brill-Noether strata: 
\begin{thm}[{\cite[Proposition 3.31, Proposition 3.32]{ny11b}}] \label{thm:bn}
The morphisms 
\begin{align*}
&\xi_{m, i}^- \colon M^m(v)_{i} \to M^{m, m+1}(v)_{i}, \\
&\xi_{m, i}^+ \colon M^{m+1}(v)_{i} \to M^{m, m+1}(v)_{i}
\end{align*}
are identified with the morphisms 
\begin{align*}
&\Gr\left(i, \mcE xt^1_{p}(\mcE'_{i}, q^*\mcO_{C}(-m-1)) \right) 
\to M^{m, m+1}(v)_{i}, \\
&\Gr\left(i, \mcE xt^1_{p}(q^*\mcO_{C}(-m-1), \mcE'_{i}) \right) 
\to M^{m, m+1}(v)_{i}, 
\end{align*}
respectively. 
In particular, every fiber of the morphisms $\xi_{m}^{\pm}$ 
is the Grassmannian variety. 
\end{thm}

\section{Birational geometry of moduli spaces} \label{sec:bir}
In this section, we will prove that the diagram (\ref{eq:nyintro}) 
realizes the MMP. 
The key ingredient is 
to compute the normal bundles 
of the fibers of $\xi_{m}^{\pm}$, 
following the arguments of 
Ellingsrud-G\"{o}ttsche \cite{eg95} 
and Friedman-Qin \cite{fq95}.
We keep the notations as in the previous section. 
Fix integers $m, i \in \mathbb{Z}_{\geq 0}$. 
Let 
\begin{equation} \label{eq:univpm}
\mcE^- \in \Coh(\hatP^2 \times M^{m}(v)), \quad 
\mcE^+ \in \Coh(\hatP^2 \times M^{m+1}(v))
\end{equation}
be the universal families, 
and let 
\begin{equation} \label{eq:mcWpm}
\mcW^-_{i}:=\mcE xt^1_{p}(\mcE'_{i}, q^*\mcO_{C}(-m-1)), \quad 
\mcW^+_{i}:=\mcE xt^1_{p}(q^*\mcO_{C}(-m-1), \mcE'_{i}) 
\end{equation}
be vector bundles on $M^{m, m+1}(v)_{i}$. 
We consider the Grassmannian bundles 
\begin{equation} \label{eq:grbdlpm}
\begin{aligned}
&\pi^- \colon G^-_{i}:=\Gr\left(i, \mcW^-_{i} \right) 
\to M^{m, m+1}(v)_{i}, \\
&\pi^+ \colon G^+_{i}:=\Gr\left(i, \mcW^+_{i} \right) 
\to M^{m, m+1}(v)_{i}, 
\end{aligned}
\end{equation}
which are isomorphic to the Brill-Noether loci 
$M^m(v)_{i}, M^{m+1}(v)_{i}$, respectively, 
by Theorem \ref{thm:bn}. 
On $G^{\pm}_{i}$, we have the following tautological sequences: 
\begin{equation} \label{eq:tautogr}
0 \to \mcS^{\pm}_{i} \to \pi^{\pm*}\mcW^{\pm}_{i} \to \mcQ^{\pm}_{i} \to 0. 
\end{equation}
Let $g^{\pm} \colon \hatP^2 \times G^{\pm}_{i} \to G^{\pm}_{i}$, 
$h^{\pm} \colon \hatP^2 \times G^{\pm}_{i} \to \hatP^2$ 
be the projections. 
We start with the following lemma. 

\begin{lem}
There exists an exact sequence 
\begin{equation} \label{eq:univg}
0 \to g^{-*}\mcS^{-\vee}_{i} \otimes h^{-*}\mcO_{C}(-m-1) \to 
\mcE^-|_{G^-_{i}} \to \pi^{-*}_{X}\mcE'_{i} \to 0
\end{equation}
on $\Coh(\hatP^2 \times G^-_i)$. 
Similarly, we have 
\[
0 \to \pi^{+*}_{X}\mcE'_{i} \to \mcE^+|_{G^+_{i}} \to 
g^{+*}\mcS^+_{i} \otimes h^{+*}\mcO(-m-1) \to 0. 
\]
on $\Coh(\hatP^2 \times G^+_i)$. 

\end{lem}
\begin{proof}
Let us take an object 
$E' \in M^{m, m+1}(v)_{i}$ and 
an $i$-dimensional subspace 
$V \subset \Ext^1(E', \mcO_{C}(-m-1))$. 
Then by \cite[Proposition 4.7]{ny11c}, 
the associated universal extension 
\[
0 \to V^{\vee} \otimes \mcO_{C}(-m-1) \to E^- \to E' \to 0
\]
defines an $m$-stable sheaf $E^- \in M^m(v)$. 
Hence $\mcE^-|_{G^-_{i}}$ coincides with 
the universal extension on $\hatP^2 \times G^-_{i}$, 
and the first assertion follows. 
The proof of the second assertion is similar. 
\end{proof}

\subsection{Birational geometry of moduli spaces}
The goal of this subsection is to prove the following: 
\begin{thm}\label{thm:normal}
For any integers $m, i \in \mathbb{Z}_{\geq 0}$, 
we have isomorphisms 
\begin{align*}
&\mathcal{N}_{G^-_{i}/M^{m}(v)} \cong 
\mathcal{S}^-_{i} \otimes \pi^{-*}\mcW^+_{i}, \\
&\mathcal{N}_{G^+_{i}/M^{m+1}(v)} \cong 
\mathcal{S}^+_{i} \otimes \pi^{+*}\mcW^-_{i}. 
\end{align*}
\end{thm}
See (\ref{eq:mcWpm}), (\ref{eq:grbdlpm}), and (\ref{eq:tautogr}) 
for the notations used in the above theorem. 
We divide the proof of the theorem into several lemmas.

\begin{lem} \label{lem:natural}
There exists a natural morphism 
\[
\delta \colon \mcT_{M^m(v)}|_{G^-_{i}} 
\to \mcS^-_{i} \otimes \pi^{-*}\mcW^+_{i}. 
\]
\end{lem}
\begin{proof}
By the deformation-obstruction theory for framed sheaves 
(cf. \cite[Theorem 4.3]{bm11}), 
the tangent bundle of $M^m(v)$ is given as 
\[
\mcT_{M^m(v)} \cong 
\mcE xt^1_{p_{M}}\left(
\mcE^-, 
\mcE^-(-l_{\infty} \times M^m(v))
\right). 
\]

Now, applying the functor 
$\mcH om_{g^-}(\mcE^-|_{G^-_{i}},(-) \otimes \mcO(-l_{\infty} \times G^-_{i}))$ 
to the exact sequence (\ref{eq:univg}), 
we get the morphism 
\[
\delta_{1} \colon 
\mcE xt^1_{g^-}\left(
\mcE^-|_{G^-_{i}}, 
\mcE^-|_{G^-_{i}}(-l_{\infty} \times G^-_{i})
\right)
\to 
\mcE xt^1_{g^-}\left(
\mcE^-|_{G^-_{i}}, 
\pi^{-*}\mcE'_{i}(-l_{\infty} \times G^-_{i})
\right). 
\]

On the other hand, applying the functor 
$\mcH om_{g^-}(-, \pi^{-*}\mcE'_{i}(-l_{\infty} \times G^-_{i}))$ 
to the exact sequence (\ref{eq:univg}), 
we get the morphism 
\begin{align*}
\delta_{2} \colon 
&\mcE xt^1_{g^-}\left(
\mcE^-|_{G^-_{i}}, 
\pi^{-*}\mcE'_{i}(-l_{\infty} \times G^-_{i})
\right)
\to \\
&\quad \quad \mcE xt^1_{g^-}\left(
g^{-*}\mcS^{-\vee}_{i} \otimes h^{-*}\mcO_{C}(-m-1), \pi^{-*}\mcE'_{i}
\right), 
\end{align*}
since $\mcO_{C}(l_{\infty})=\mcO_{C}$. 
Note that we have 
\[
\mcE xt^1_{g^-}\left(
g^{-*}\mcS^{-\vee}_{i} \otimes h^{-*}\mcO_{C}(-m-1), \pi^{-*}\mcE'_{i}
\right)
\cong 
\mcS^-_i \otimes \pi^{-*}\mcW^+_{i}. 
\]
By the above arguments, we have a morphism 
\[
\delta:=\delta_{2} \circ \delta_{1} \colon 
\mcT_{M^m(v)}|_{G^-} \to \mcS^-_{i} \otimes \pi^{-*}\mcW^+_{i}. 
\]
\end{proof}

In the following lemma, we show that 
our morphism $\delta$ is surjective, 
by checking it on the fibers of the morphism 
$\pi^- \colon G^-_i \to M^{m, m+1}(v)_i$ 
in (\ref{eq:grbdlpm}). 
By abuse of notation, we also denote by 
$\delta_1, \delta_2$ 
their restrictions to the $\pi^-$-fibers. 

\begin{lem} \label{lem:surj}
Let us take an object 
$E' \in M^{m, m+1}(v)_{i}$ 
and put $W^-:=\Ext^1(E', \mcO_{C}(-m-1))$, 
$W^+:=\Ext^1(\mcO_{C}(-m-1), E')$. 
Take also an $i$-dimensional subspace 
$V \subset W^-$, 
and let $E^- \in M^m(v)$ 
be the associated universal extension. 
Then we have the following exact sequences: 
\begin{equation} \label{eq:di}
\begin{aligned}
&\Hom(V, W^-/V) 
\xrightarrow{\alpha_{1}} 
\Ext^1(E^-, E^-(-l_{\infty})) 
\xrightarrow{\delta_{1}} 
\Ext^1(E^-, E'(-l_{\infty})) \to 0, \\
&\Ext^1(E', E'(-l_{\infty})) 
\xrightarrow{\alpha_{2}} 
\Ext^1(E^-, E'(-l_{\infty})) 
\xrightarrow{\delta_{2}} 
V \otimes W^+ \to 0. 
\end{aligned}
\end{equation}
\end{lem}
\begin{proof}
By replacing the object 
$E' \in M^{m, m+1}(v)_i$ with 
$E' \otimes \mcO_{\hatP^2}(-mC) \in M^{m, m+1}(v.e^{-mC})_i$, 
we may assume $m=0$. 
Let us take an object 
$E' \in M^{0, 1}(v)_{i}$ 
and an $i$-dimensional subspace 
$V \subset \Ext^1(E', \mcO_{C}(-1))$. 
Let 
\begin{equation} \label{eq:tautoext}
0 \to V^{\vee} \otimes \mcO_{C}(-1) \to E^- \to E' \to 0
\end{equation}
be the associated universal extension. 
By applying the functor 
$\Hom(E^-, (-) \otimes \mcO(-l_\infty))$, 
we have the exact sequence 
\begin{align*}
&\Ext^1(E^-, V^{\vee} \otimes \mcO_{C}(-1)) \to 
\Ext^1(E^-, E^-(-l_\infty)) \xrightarrow{\delta_{1}} 
\Ext^1(E^-, E'(-l_\infty)) \\
\to &\Ext^2(E^-, V^{\vee} \otimes \mcO_{C}(-1)). & & 
\end{align*}
Note that we have used the fact that 
$\mcO_C(-l_\infty)=\mcO_C$. 
By Serre duality and the $0$-stability of $E^-$, 
we have the vanishing 
\begin{align*}
\Ext^2(E^-, \mcO_{C}(-1)) 
\cong \Hom(\mcO_{C}, E^-)^{\vee} 
=0 
\end{align*}
and hence $\delta_{1}$ is surjective. 
Furthermore, by applying the functor 
$\Hom(-, \mcO_{C}(-1))$ 
to the exact sequence (\ref{eq:tautoext})
we have the exact sequence 
\[
0 \to V \to W^- \to \Ext^1(E^-, \mcO_{C}(-1)) 
\to \Ext^1(V^{\vee} \otimes \mcO_{C}(-1), \mcO_{C}(-1))=0. 
\]
Hence we have 
\begin{align*}
\Ext^1(E^-, V^{\vee} \otimes \mcO_{C}(-1)) 
\cong \Hom(V, W^-/V) 
\end{align*}
as required. 

Similarly, applying the functor 
$\Hom(-, E'(-l_\infty))$ 
to the exact sequence (\ref{eq:tautoext}), 
we obtain 
\begin{align*}
&\Ext^1(E', E'(-l_{\infty})) 
\xrightarrow{\alpha_{2}} 
\Ext^1(E^-, E'(-l_{\infty})) 
\xrightarrow{\delta_{2}} 
V \otimes W^+ \\
\to &\Ext^2(E', E'(-l_\infty)). 
\end{align*}
By Lemma \ref{lem:vanobstr} below, 
we have the vanishing 
$\Ext^2(E', E'(-l_\infty))=0$ 
and hence $\delta_2$ is surjective. 
\end{proof}

\begin{lem} \label{lem:vanobstr}
Let $F$ be a $0$-stable sheaf. 
Then we have the vanishing 
\[
\Ext^2(F, F(-l_\infty))=0. 
\]
\end{lem}
\begin{proof}
By Serre duality, we have 
\[
\Ext^2(F, F(-l_\infty))
\cong \Hom(F(-l_\infty), F \otimes \omega_{\hatP^2})^\vee. 
\]

Moreover, as in the proof of 
\cite[Lemma 3.6]{ny11b}, 
we have an injection 
\[
\Hom_{\hatP^2}(F(-l_\infty), F \otimes \omega_{\hatP^2}) 
\hookrightarrow 
\Hom_{\bP^2}\left( 
(f_*F)^{\vee \vee}(-l_\infty), (f_*F)^{\vee \vee} \otimes \omega_{\bP^2}
\right), 
\]
where $f \colon \hatP^2 \to \bP^2$ 
is the blow-up morphism. 
Putting 
$G:=(f_*F)^{\vee \vee}$, 
it is enough to show 
\[
\Hom(G(-l_\infty), G \otimes \omega_{\bP^2})
=\Hom(G, G(-2))=0. 
\]

By applying the functor 
$\Hom(G, -)$ 
to the exact sequence 
\begin{equation} \label{eq:GexP2}
0 \to G(-3) \to G(-2)
\to G(-2)|_{l_\infty} \to 0, 
\end{equation}
we obtain 
\begin{equation} \label{eq:GonP2}
0 \to \Hom\left(G, G(-3) \right) 
\to \Hom\left(G, G(-2) \right) 
\to \Hom\left(G, G(-2)|_{l_\infty} \right). 
\end{equation}

On the other hand, as $G$ is a framed sheaf on $\bP^2$, 
we have $G|_{l_\infty} \cong \mcO_{\bP^2}^{\oplus \ch_0(G)}$. 
Hence we have 
\begin{align*}
\Hom\left(G, G(-2)|_{l_\infty} \right) 
&\cong \Hom\left(G|_{l_\infty}, G(-2)|_{l_\infty} \right) \\
&\cong \Hom\left(\mcO_{l_\infty}, \mcO_{l_\infty}(-2) \right)^{\oplus \ch_0(G)^2} \\
&=0. 
\end{align*}
Combining with the exact sequence (\ref{eq:GonP2}), 
we conclude that 
\[
\Hom(G, G(-2)) \cong 
\Hom(G, G(-3)). 
\]

By tensoring the exact sequence (\ref{eq:GexP2}) 
with $\mcO_{\bP^2}(-i)$ and repeating the above argument, 
we can inductively prove the isomorphism 
\[
\Hom(G, G(-2)) \cong 
\Hom(G, G(-i-3)) 
\]
for all $i \geq 0$. 
By Serre duality, we obtain 
\[
\Hom(G, G(-i-3)) \cong 
H^2(G \otimes G^\vee(i))^\vee, 
\]
and the right hand side vanishes for sufficiently large $i > 0$, 
since $\mcO_{\bP^2}(1)$ is ample. 
We conclude that 
$\Hom(G, G(-2))=0$ as required. 
\end{proof}

We also need the following: 
\begin{lem}\label{lem:dimext}
We have the equalities 
\begin{align*}
&\dim M^m(v)
=\dim G^-_{i} 
+ \rk\left(
\mcS^-_{i} \otimes \pi^{-*}\mcW^+_{i} 
\right), \\
&\dim M^{m+1}(v)
=\dim G^+_{i} 
+ \rk\left(
\mcS^+_{i} \otimes \pi^{+*}\mcW^-_{i} 
\right). 
\end{align*}
\end{lem}
\begin{proof}
We only prove the first equality. 
By the dimension formula for the framed moduli space 
in Theorem \ref{thm:adhm} (1), 
we have 
\begin{equation} \label{eq:dimMm}
\begin{aligned}
\dim M^m(v)
&=d_{\infty}(d_{0}+d_{1})-(d_{0}-d_{1})^2 \\
&=-2r\ch_2-k^2, 
\end{aligned}
\end{equation}
where the vector $(d_{0}, d_{1}, d_{\infty})$ 
is defined as in (\ref{eq:dim}). 
On the other hand, 
for any object $E' \in M^{m, m+1}(v)_{i}$, 
we have 
\begin{equation} \label{eq:dimG-}
\begin{aligned}
\dim G^-_{i} 
&= \dim M^{m, m+1}(v)_{i} 
+ i\left(\ext^1(E', \mcO_{C}(-m-1))-i \right) \\
&= \dim M^m(v-ic_{m}) 
-i\left(\chi(E', \mcO_{C}(-m-1))+i \right), 
\end{aligned}
\end{equation}
where the first equality follows from the fact that 
$\pi^- \colon G^-_i \to M^{m, m+1}(v)_{i}$ is a Grassmannian bundle 
defined as in (\ref{eq:grbdlpm}). 
For the second equality, 
first observe that 
we have 
$\chi(E', \mcO_{C}(-m-1))=-\ext^1(E', \mcO_{C}(-m-1))$ 
since $E'$ is $m$-stable and $(m+1)$-stable. 
Moreover, the inclusion 
\[
M^{m, m+1}(v)_{i}=M^m(v-ic_m) \cap M^{m+1}(v-ic_m) 
\subset M^m(v-ic_m) 
\]
is an open immersion 
and hence the second equality in (\ref{eq:dimG-}) holds. 
Similarly, by using the $m$-stability and $(m+1)$-stability of $E'$, 
we have 
\begin{equation} \label{eq:dimS-W+}
\begin{aligned}
\rk\left(
\mcS^-_{i} \otimes \pi^{-*}\mcW^+_{i} 
\right)
&=i\ext^1(\mcO_{C}(-m-1), E') \\
&=-i \cdot \chi(\mcO_{C}(-m-1), E') 
\end{aligned}
\end{equation}
(see (\ref{eq:mcWpm}) and (\ref{eq:tautogr}) 
for the definitions of $\mcW^+_i$ and $\mcS^-_i$).  
Combining the equalities 
(\ref{eq:dimMm}), (\ref{eq:dimG-}) and (\ref{eq:dimS-W+}), 
we obtain 
\begin{align*}
&\quad \dim G^-_i +\rk\left(
\mcS^-_{i} \otimes \pi^{-*}\mcW^+_{i} 
\right) \\
&=\dim M^m(v-ic_m)
-i\left(
i+\chi(E', \mcO_{C}(-m-1))
+\chi(\mcO_{C}(-m-1), E')
\right) \\
&=-2r\left(\ch_2+i\left(m+\frac{1}{2} \right) \right)-(k+i)^2 \\
&\quad -i^2-i\left(\chi(E', \mcO_{C}(-m-1))
+\chi(\mcO_{C}(-m-1), E') \right) \\
&=-2r\ch_2-k^2=\dim M^m(v), 
\end{align*}
where the second equality holds since we have 
\[
v-ic_m=v-i\ch(\mcO_C(-m-1))
=\left(r, -(k+i)C, \ch_2+i\left(m+\frac{1}{2} \right) \right), 
\]
and for the third equality, 
we use the Serre duality 
$\chi(E', \mcO_C(-m-1))=\chi(\mcO_C(-m), E')$, 
together with the Riemann-Roch theorem 
\[
-\chi(\mcO_C(-m), E')
=rm+k+i. 
\]
\end{proof}

Now we begin the proof of Theorem \ref{thm:normal}. 
\begin{proof}[Proof of Theorem \ref{thm:normal}]
We only prove the first assertion. 
By Lemma \ref{lem:natural} and Lemma \ref{lem:surj}, 
we have a surjective morphism 
\[
\delta \colon \mcT_{M^m(v)}|_{G^-_{i}} 
\to \mcS^-_{i} \otimes \pi^{-*}\mcW^+_{i}. 
\]
We need to show the isomorphism 
\[
\mcT_{G^-_{i}} \cong \ker(\delta). 
\]
Let the notations be as in Lemma \ref{lem:surj}. 
By Lemma \ref{lem:dimext}, 
the vector spaces 
$T_{E^-}G^-_{i}$ and $\ker(\delta_{E^-})$ 
are of the same dimension. 
Hence it is enough to show that the composition 
\[
T_{E^-}G^-_{i} \hookrightarrow T_{E^-}M^m(v) 
\xrightarrow{\delta_{E^-}} 
\Ext^1(V \otimes \mcO_{C}(-m-1), E') 
\]
is zero. 
Indeed, if this is the case, then 
we have $T_{E^-}G^-_{i}=\ker(\delta_{E^-})$, 
which induces a surjection 
$\mcT_{G^-_{i}} \twoheadrightarrow \ker(\delta)$ 
between torsion free sheaves of the same rank: 
it should be an isomorphism. 

We have the exact sequence 
\[
0 \to \Hom(V, W^-/V) \to T_{E^-}G^-_{i} \to T_{E'}M^{m, m+1}(v)_{i} \to 0. 
\]
We can see that the composition 
\[
\Hom(V, W^-/V) \to T_{E^-}G^-_{i} \to T_{E^-}M^m(v) 
\]
coincides with the morphism $\alpha_{1}$ 
in the exact sequence (\ref{eq:di}). 
In particular, it becomes zero after composing with $\delta_{1}$. 
Hence the morphism 
\[
T_{E^-}G^-_{i} \hookrightarrow T_{E^-}M^m(v) 
\xrightarrow{\delta} V \otimes W^+
\] 
factors through $T_{E'}M^{m, m+1}(v)_{i}$. 
Similarly, the morphism 
$T_{E'}M^{m, m+1}(v)_{i} \to V \otimes W^+$ 
coincides with $\delta \circ \alpha_{2}$, 
which is zero by the second exact sequence in (\ref{eq:di}). 
We conclude that 
$T_{E^-}G^-_{i}=\ker(\delta_{E^-})$ 
as required. 
\end{proof}

Now we have the following theorem: 
\begin{thm} \label{mmp}
Fix a Chern character of the form 
$v=(r, 0, \ch_{2}) \in H^{2*}(\hatP^2, \mathbb{Q})$. 
Then the diagram (\ref{eq:nyintro}) is a minimal model program 
for the moduli space $M_{\hatP^2}(v)$ 
of framed torsion free sheaves 
on the blow-up $\hatP^2$. 
The program ends with the minimal model, 
the moduli space $M_{\bP^2}(r, 0, \ch_{2})$ 
of framed torsion free sheaves on $\bP^2$, 
which is a hyper-K$\ddot{a}$hler manifold. 
\end{thm}
\begin{proof}
We claim that for each $m \in \mathbb{Z}_{\geq 0}$, 
the morphism $\xi_{m}^+$ 
(resp. $\xi_{m}^-$) 
is a $K$-negative (resp. $K$-positive) contraction 
(cf. Definition \ref{def:k-neg/pos}). 
By Lemma \ref{lem:KofY} and 
Theorem \ref{thm:normal}, 
it is enough to show the inequality 
\[
\rk \mcW^+_{i} > \rk \mcW^-_{i} 
\]
(see (\ref{eq:mcWpm}) for the difinitions of $\mcW^\pm_i$), 
which is equivalent to the inequality 
\[
\ext^1(\mcO_{C}(-m-1), E') > 
\ext^1(E', \mcO_{C}(-m-1)). 
\]
Now the assertion directly follows 
from the Riemann-Roch theorem. 
Explicitly, we have 
\begin{equation} \label{eq:difference}
\begin{aligned}
&\quad \ext^1(\mcO_{C}(-m-1), E') 
-\ext^1(E', \mcO_{C}(-m-1)) \\
&=\chi(E', \mcO_{C}(-m-1)) 
-\chi(\mcO_{C}(-m-1), E') \\
&=r. 
\end{aligned}
\end{equation}

To see that the moduli space 
$M_{\bP^2}(r, 0, \ch_2)$ is hyper-K{\"a}hler, 
recall that the space $M_{\bP^2}(r, 0, \ch_2)$ is 
isomorphic to Nakajima's quiver variety associated with 
the quiver with one vertex and one loop \cite[Chapter 2]{nak99}. 
By the general fact that 
Nakajima's quiver varieties are hyper-K{\"a}hler \cite{nak94, gin09}, 
so is the variety $M_{\bP^2}(r, 0, \ch_2)$. 
\end{proof}

From the arguments above, 
we can also deduce the following result: 
\begin{prop}
Let us take a Chern character 
$v=(r, 0, \ch_2) \in H^{2*}(\hatP^2, \mathbb{Q})$. 
The following statements hold. 
\begin{enumerate}
\item For every integer $m \geq 1$, 
the morphism $\xi_{m}^{\pm}$ is a small contraction. 
\item The morphism $\xi_{0}^{+}$ is 
a divisorial contraction. 
\end{enumerate}
\end{prop}
\begin{proof}
By Lemma \ref{lem:dimext} 
and the equation (\ref{eq:difference}), 
we have the inequality 
$\dim G^+_{i} > \dim G^-_{i}$. 
Hence it is enough to estimate the dimension of $G^+_{i}$. 
By Lemma \ref{lem:dimext}, 
we have 
\begin{align*}
\codim (G^+_{i}, M^{m+1}(v) ) 
&=\rk\left(
\mcS^+_{i} \otimes \pi^{+*}\mcW^-_{i} 
\right) \\
&=i\ext^1(E', \mcO_C(-m-1))
=rim + i^2. 
\end{align*}

It follows that for each $m \geq 1$, 
the morphism $\xi_{m}^{\pm}$ 
is a small contraction, 
while $\xi_{0}^+$ is a divisorial contraction. 
\end{proof}
\subsection{Fully faithful embedding between derived categories}
As an application of Theorem \ref{thm:normal}, 
we show the existence of the fully faithful embedding 
between the derived categories 
of the moduli spaces $M^m(v), M^{m+1}(v)$. 
Let $W^{m}:=M^m(v) \times_{M^{m, m+1}(v)} M^{m+1}(v)$ 
be the fiber product. 
\begin{thm} \label{derived}
Let us fix a Chern character 
$v=(r, 0, \ch_{2}) \in H^{2*}(\hatP^2, \mathbb{Q})$. 
Then for each integer $m \in \mathbb{Z}_{\geq 0}$, 
we have the fully faithful functor 
\[
\Phi=\Phi^{\mcO_{W^m}} \colon D^b(M^m(v)) \hookrightarrow D^b(M^{m+1}(v))
\]
whose Fourier-Mukai kernel is $\mcO_{W^m}$. 
In particular, we have a fully faithful embedding 
\[
D^b(M_{\bP^2}(r, 0, \ch_{2})) \hookrightarrow 
D^b(M_{\hatP^2}(r, 0, \ch_{2})). 
\]
\end{thm}
\begin{proof}
By Lemma \ref{lem:formal}, 
we can reduce the statement to the formal completion 
at a point $[E'] \in M^{m, m+1}(v)$. 

We claim that the diagram (\ref{eq:nyintro}) 
is formally locally isomorphic to the Grassmannian flip 
appered in Section \ref{subsec:gflip}, 
by using Lemma \ref{lem:formal}. 
Let us take an object 
$[E'] \in M^{m, m+1}(v)_{i}$ 
and put 
$U:=\Ext^1(E'(-l_{\infty}), E')$, 
$W^-:=\Ext^1(E', \mcO_{C}(-m-1))$, 
$W^+:=\Ext^1(\mcO_{C}(-m-1), E')$. 
Recall that the fibers of $\xi^{\pm}_{m}$ 
are the Grassmannian varieties 
\[
(\xi^{\pm}_{m})^{-1}([E'])=G^{\pm}(E'):=\Gr(i, W^{\pm}). 
\]
By Theorem \ref{thm:normal}, 
their normal bundles are given as 
\begin{align*}
&N_{G^-(E')/M^m(v)}=S^- \otimes W^+ \oplus \mcO^{\oplus \dim U}, \\
&N_{G^+(E')/M^{m+1}(v)}=S^+ \otimes W^- \oplus \mcO^{\oplus \dim U}. 
\end{align*}
Here, $S^{\pm}$ denotes the tautological subbundles 
on $G^{\pm}(E')$. 
First note that the conormal bundle 
$N^\vee_{G^+(E')/M^{m+1}(v)}$ is nef 
as it is globally generated. 
Combining with Lemma \ref{lem:vanGrS}, 
we can apply Lemma \ref{lem:formal} to 
$G^+(E) \subset M^{m+1}(v)$. 
Moreover, its flip is unique by \cite[Corollary 6.4]{km98}. 
We conclude that 
the formal completion of the diagram (\ref{eq:nyintro}) 
is isomorphic to the formal completion of the diagram 
\begin{equation} \label{eq:determ}
\xymatrix{
&Y^- \times U \ar[rd] 
& 
&Y^+ \times U \ar[ld] \\
& &Z \times U &
}
\end{equation}
at a point $(0, 0) \in Z \times U$. 
Here, varieties $Y^{\pm}$ and $Z$ are defined as 
\begin{align*}
&Y^-:=\Tot_{G^-(E')}(S^- \otimes W^+), \\
&Y^+:=\Tot_{G^+(E')}(S^+ \otimes W^-), \\
&Z:=\left\{
a \in \Hom((W^-)^\vee, W^+) : 
\rk a \leq i
\right\} 
\end{align*}
(see Section \ref{subsec:gflip}). 
By Theorem \ref{thm:blvdb}, we have the fully faithful functor 
\[
\Phi^{loc} \colon D^b(Y^-) \hookrightarrow D^b(Y^+)
\]
whose Fourier-Mukai kernel is 
the structure sheaf of the fiber product 
$Y^- \times_{Z} Y^+$. 
Hence globally, the functor 
$\Phi \colon D^b(M^m(v)) \to D^b(M^{m+1}(v))$ 
is fully faithful. 
\end{proof}

\subsection{Projective case}
Let $S$ be a smooth projective surface, 
$f \colon \hat{S} \to S$ be the blow-up at a point. 
Let $H$ be an ample divisor on $S$. 
In this setting, 
we can consider the $m$-stability 
for coherent sheaves $E$ on $\hat{S}$ 
(cf. \cite{ny11b}), 
by replacing the condition (3) in Definition \ref{defin:m-st} 
with 
\begin{enumerate}
\item[(3)'] $f_{*}(E(-mC))$ is $\mu_{H}$-stable. 
\end{enumerate}

Let us fix a cohomology class 
$w=(w_{0}, w_{1}, w_{2}) \in H^{2*}(S, \mathbb{Q})$ 
which is in the image of the Chern character map, 
and $v:=f^*w \in H^{2*}(\hat{S}, \mathbb{Q})$. 
Assume the following conditions hold: 
\begin{itemize}
\item $K_{S}.H <0$, 
\item $w_{0} >0$, 
\item $\gcd(w_{0}, H.w_{1})=1$. 
\end{itemize} 
Then by \cite[Corollary 3.7]{ny11b}, 
the moduli space $M^m(v)$ of $m$-stable sheaves 
with Chern character $v$ is smooth. 
Moreover, the analogous results as 
Theorem \ref{thm:adhm} and Theorem \ref{thm:bn} hold. 
More precisely, the moduli spaces $M^m(v)$ 
satisfy the following properties: 
\begin{itemize}
\item $M^m(v)$ has the Brill-Noether stratification as in Section \ref{subsec:bn} 
by \cite[Propositions 3.31 and 3.32]{ny11b}.  
\item For $m=0$, we have an isomorphism 
$f_* \colon M^0(v) \xrightarrow{\cong} M_H(w)$, 
where $M_H(v)$ denotes 
the moduli space of Gieseker stable sheaves on $S$ 
with respect to the polarization $H$, 
by \cite[Proposition 3.3]{ny11b}. 
\item For $m \gg 0$, we have an isomorphism 
$M^m(v) \cong M_{f^*H-\epsilon C}(v)$, 
where $M_{f^*H-\epsilon C}(v)$ denotes 
the moduli space of Gieseker stable sheaves on $\hat{S}$ 
with respect to the polarization $f^*H-\epsilon C$ 
for sufficiently small $\epsilon >0$, 
by \cite[Proposition 3.37]{ny11b}. 
\end{itemize}

In the projective setting, 
we can also prove the results 
in the previous subsections 
by using the properties listed above. 
We omit the proof, 
since the arguments are very similar. 

\begin{thm} \label{thm:proj}
Let the notations be as above. 
Take a sufficiently small positive real number 
$0 <\epsilon \ll 1$. 
Then an MMP for the moduli space $M_{f^*H-\epsilon C}(v)$ 
of Gieseker stable sheaves on $\hat{S}$ 
is reduced to an MMP for the moduli $M_{H}(w)$ on $S$. 
Furthermore, there exists a fully faithful embedding 
\[
D^b(M_{H}(w))
\hookrightarrow 
D^b(M_{f^*H-\epsilon C}(v))
\] 
between their derived categories. 
\end{thm}

\section{Examples} \label{sec:ex}
In this section, we give some explicit examples. 

\subsection{Simplest example}
As the first example, we consider the case when 
the Chern character is $(r, 0, -1)$ with $r \geq 1$. 
In this case, the corresponding quiver representation is 
\begin{equation} \label{eq:rep11r}
\xymatrix{
&\bC \ar@<-0.8ex>[rr]_-d & 
&\bC \ar@<-1.0ex>[ll]_{B_{1}, B_{2}} \ar@<-0.1ex>[ll] \ar@<+0.2ex>[ld]^-j \\
& &\bC^r \ar@<+0.2ex>[lu]^-i. & 
}
\end{equation}

We begin with the following easy observation: 
\begin{lem}
The moduli spaces 
$M^m(r, 0, -1)$ 
and $M^{m+1}(r, 0, -1)$ 
are isomorphic for $m \geq 1$. 
\end{lem}
\begin{proof}
Let $m \geq 0$ and $E \in M^m(r, 0, -1)$ be an $m$-stable sheaf. 
Assume that $E$ is not $(m+1)$-stable. 
Then we must have 
$\Hom(\mcO_C(-m-1), E) \neq 0$, 
and obtain the exact sequence 
\[
0 \to \Hom(\mcO_C(-m-1), E) \otimes \mcO_C(-m-1) \xrightarrow{ev} 
E \to E' \to 0
\]
as in (\ref{eq:destabE-}), 
for some $m$-stable and $(m+1)$-stable sheaf $E'$. 

On the other hand, by \cite[Proposition 5.3]{ny11a}, 
the sheaf $\mcO_C(-m-1)$ corresponds to the quiver representation 
\begin{equation} \label{eq:repOcm}
\xymatrix{
&\bC^m \ar@<-0.8ex>[rr]_-{0} & 
&\bC^{m+1} \ar@<-1.0ex>[ll]_{B_{1}, B_{2}} \ar@<-0.1ex>[ll] \ar@<+0.2ex>[ld] \\
& &0 \ar@<+0.2ex>[lu], & 
}
\end{equation}
with $B_1=(1_m, 0), B_2=(0, 1_m)$, 
where $1_m$ denotes the $(m \times m)$ identity matrix. 
Hence the representation (\ref{eq:repOcm}) cannot 
have an injection into the representation (\ref{eq:rep11r}) 
for $m >0$. 
This shows the inclusion 
$M^m(r, 0, -1) \subset M^{m+1}(r, 0, -1)$ for $m >0$. 
A similar argument shows that the opposite inclusion 
$M^{m+1}(r, 0, -1) \subset M^m(r, 0, -1)$. 
\end{proof}

Furthermore, we can describe these moduli spaces explicitly: 
\begin{lem} \label{lem:1stexam}
We have the diagram 
\[
\xymatrix{
&M^0(r, 0, -1) \ar@{=}[rd] 
& 
&M^1(r, 0, -1) \ar[ld] \\
& &M^{0, 1}(r, 0, -1) & 
}
\]
and isomorphisms 
\begin{align*}
&M_{\bP^2}(r, 0, -1) \cong 
M^0(r, 0, -1) \cong 
\Tot_{\bP^{(r-1)}}\left(
\mcO^{\oplus 2} \oplus \Omega
\right), \\ 
&M_{\hatP^2}(r, 0, -n) \cong 
M^1(r, 0, -1) \cong \Bl_{\bP^{(r-1)}}M^0(r, 0, -1), 
\end{align*}
where $\bP^{(r-1)}$ is embedded into $M^0(r, 0, -1)$ 
as the zero section. 
\end{lem}
Note that when $r=1$, we recover 
the blow-up morphism $\hatC^2 \to \bC^2$ 
as the moduli spaces (cf. \cite[Theorem 2.16]{ny11a}). 

Before starting the proof, 
let us recall the ADHM description 
of the framed sheaves on $\bP^2$ 
(see \cite[Section 1]{ny11a} and \cite[Capter 2]{nak99} for the details). 
Let $V, W$ be vector spaces. 
An {\it ADHM data} is the data 
$X:=(B_{1}, B_{2}, i, j)$, 
where $B_{\alpha} \in \End(V)$, 
$i \in \Hom(W, V)$, and 
$j \in \Hom(V, W)$ satisfying the relation 
\begin{equation} \label{eq:adhmP2}
[B_{1}, B_{2}]+ij=0. 
\end{equation}
An ADHM data $X=(B_1, B_2, i, j)$ is called 
{\it stable} if there is no proper subspace $T \subset V$ 
such that $B_\alpha(T) \subset T$ for $\alpha=1, 2$, and 
$\Image(i) \subset T$. 
We have the moduli space of stable ADHM data 
as the quotient of the stable locus inside the affine space 
\[
\End(V)^{\times 2} \times \Hom(W, V) \times \Hom(V, W) 
\]
modulo the natural $\GL(V)$-action. 

Then the moduli space $M_{\bP^2}(r, 0, \ch_2)$ 
of torsion free framed sheaves on $\bP^2$ 
is isomorphic to the moduli space of stable ADHM data 
with $\dim V=-\ch_2, \dim W=r$ 
(cf. \cite[Chapter 2]{nak99})

\begin{proof}[Proof of Lemma \ref{lem:1stexam}]
Let us put $\dim V=1, \dim W=r$, 
and let $X=(B_1, B_2, i, j)$ be a stable ADHM data. 
Since $\dim V=1$, the relation (\ref{eq:adhmP2}) 
becomes $ij=0$, and the stability condition becomes 
$i \neq 0$. 
Hence the stable locus is given as 
\[
\left(
\End(V)^{\times 2} \times (W^{\vee} \setminus \{0\}) \times W) 
\right)
\cap \mu^{-1}(0), 
\]
where 
\[
\mu \colon \End(V)^{\times 2} \times W^\vee \times W \to \bC, \quad 
(B_{1}, B_{2}, i, j) \mapsto ij
\]
is the moment map. 
As $\GL(V)=\bC^*$ acts trivially on $\End(V)$, 
and on $W$ by weight $-1$, 
we have 
\[
M_{\bP^2}(r, 0, -1) \cong 
\Tot_{\bP^{(r-1)}}\left(
\mcO^{\oplus 2} \oplus \Omega
\right). 
\]

Next we determine the variety $M^1(r, 0, -1)$. 
Recall that we have an isomorphism 
(cf. \cite[Proposition 7.4]{ny11a})
\begin{equation} \label{eq:isomadhm}
\phi \colon 
M^0(r, 0, -1) \xrightarrow{\cong} M_{\bP^2}(r, 0, -1), \quad 
(B_{1}, B_{2}, d, i, j) \mapsto 
(dB_{1}, dB_{2}, di, j). 
\end{equation}
The locus blown-up by $\xi_{0}^+$ 
is given by 
\begin{equation} \label{eq:bn11r}
M^0(r, 0, -1)_1=\left\{
E \in M^0(r, 0, -1) : 
\hom(\mcO_{C}(-1), E)=1 
\right\}. 
\end{equation}
As the sheaf $\mcO_C(-1)$ corresponds to 
the representation (\ref{eq:repOcm}) with $m=0$, 
we see that the locus (\ref{eq:bn11r}) 
coincides with the subvariety 
\[
\bP^{(r-1)} \cong L:=\left(
B_{1}=B_{2}=j=0
\right)
\subset M_{\bP^2}(r, 0, -1) 
\]
under the isomorphism (\ref{eq:isomadhm}), 
which is nothing but the zero section. 
\end{proof}

\subsection{Hilbert scheme of points}
In this subsection, 
we consider the moduli spaces 
with Chern character $(1, 0, -n)$. 
In this case, the diagram (\ref{eq:nyintro}) 
connects the Hilbert schemes of points 
$\Hilb^n(\bC^2)$ 
and $\Hilb^n(\hatC^2)$. 
We first analyze the stability of ideal sheaves 
$I_{Z} \in \Hilb^n(\hatC^2)$ 
(it is an ideal sheaf 
of a length $n$ subscheme 
$Z \subset \hatP^2$ 
with $Z \cap l_{\infty} = \emptyset$). 
\begin{lem} \label{lem:ideal}
Let us take a point $I_{Z} \in \Hilb^n(\hatC^2)$ 
and let $k$ be a length of $Z \cap C$. 
Then $I_{Z}$ is $k$-stable but not $(k-1)$-stable. 
Furthermore, its destabilizing sequence 
for $(k-1)$-stability is given as 
\[
0 \to I_{W}(-C) \to I_{Z} \to \mcO_{C}(-k) \to 0 
\]
for some length $(n-k)$ zero dimensional subscheme 
$W \subset \hatP^2$ with $W \cap C = \emptyset$. 
\end{lem}
\begin{proof}
The first statement follows from the proof of 
\cite[Lemma 6.1]{kos17}. 
For the second statement, 
let $\mcO_{W}$ be the kernel of the surjection 
$\mcO_{Z} \to \mcO_{Z \cap C}$. 
Then we have the following diagram. 
\[
\xymatrix{
& &0 \ar[d] &0 \ar[d] &0 \ar[d] & \\
&0 \ar[r] 
&I_{W}(-C) \ar[r] \ar[d] 
&\mcO(-C) \ar[r] \ar[d] 
&\mcO_{W} \ar[r] \ar@{=}[d] &0 \\
&0 \ar[r] 
&I_{Z} \ar[r] \ar[d] 
&I_{Z \cap C} \ar[r] \ar[d] 
&\mcO_{W} \ar[r] &0 \\
& &\mcO_{C}(-k) \ar@{=}[r] \ar[d] &\mcO_{C}(-k) \ar[d] & & \\
& &0 &0 & &
}
\]
Furthermore, the sheaf $I_{W}(-C)$ is 
$0$-stable since $W \cap C=\emptyset$. 
Hence the second assertion follows. 
\end{proof}

\subsubsection{}
When $n=2$, we have the following diagram: 
\[
\xymatrix{
& &M^1(1, 0, -2) \ar[rd]_{\xi_{1}^-} \ar[ld]^{\xi_{0}^+} & 
&\Hilb^2(\hatC^2) \ar[ld]^{\xi_{1}^+} \\
&\Hilb^2(\bC^2) & &M^{0, 1}(1, 0, -2). & 
}
\]
The properties of the diagram are summarized as follows
(cf. \cite[Theorem 1.4]{kos17}): 
\begin{itemize}
\item $\xi_{1}^+$ contracts $\Hilb^2(C) \cong \bP^2$, 
and $\xi_{1}^-$ contracts $\bP^1$. 
\item The birational map 
$\Hilb^2(\hatC^2) \dashrightarrow M^1(1, 0, -2)$ 
is a standard flip. 
\item $\xi_{0}^+$ is the blow-up at 
the codimension two subvariety 
\[
\left\{
I_{Y} \in \Hilb^2(\bC^2) : Y \ni 0
\right\}
\cong \hatC^2. 
\]
\end{itemize}
The derived category 
$D^b(\Hilb(\hatC^2))$ 
has the following semi-orthogonal decomposition: 
\begin{align*}
D^b(\Hilb^2(\hatC^2))
&=\left\langle
D^b(\pt), D^b(M^1(1, 0, -2) 
\right\rangle \\
&=\left\langle
D^b(\pt), D^b(\hatC^2), D^b(\Hilb^2(\bC^2)) 
\right\rangle \\
&=\left\langle
D^b(\pt), D^b(\pt), D^b(\bC^2), D^b(\Hilb^2(\bC^2)) 
\right\rangle. 
\end{align*}

\subsubsection{}
Next we consider the case when $n=3$. 
First let us analyze geometry of 
\begin{equation} \label{eq:m3}
M^2(1, 0, -3) \xrightarrow{\xi_{2}^-} 
M^{2, 3}(1, 0, -3) \xleftarrow{\xi_{2}^+} 
\Hilb^3(\hatC^2). 
\end{equation}
By Lemma \ref{lem:ideal}, we have 
$\Exc(\xi_{2}^+)=\Hilb^3(C) \cong \bP^3$,  
\[
\Exc(\xi_{2}^-) \cong 
\bP(\Ext^1(\mcO(-C), \mcO_{C}(-3))) 
=\bP^2, 
\]
and the diagram (\ref{eq:m3}) 
is a standard flip. 

Next we analyze the geometry of the morphisms 
$\xi_{1}^{\pm}$. 
We have just seen that 
\[
M^2(1, 0, -3)=
\left(
\Hilb^3(\hatC^2) \setminus \Hilb^3(C)
\right) \cup 
\bP\Ext^1(\mcO(-C), \mcO_{C}(-3)). 
\]
Take a non-trivial extension 
\begin{equation} \label{eq:2stext}
0 \to \mcO_{C}(-3) \to E_{2} \to \mcO(-C) \to 0, 
\end{equation} 
which defines a $2$-stable sheaf 
$[E_{2}] \in M^2(1, 0, -3)$. 
We claim that we have the equality 
$\hom(E_{2}, \mcO_{C}(-2)) = 1$. 
Indeed, by applying the functor 
$\Hom(-, \mcO_{C}(-2))$ 
to the exact sequence (\ref{eq:2stext}), 
we have the exact sequence 
\[
0 \to \Hom(E_{2}, \mcO_{C}(-2)) 
\to \Hom(\mcO_{C}(-3), \mcO_{C}(-2)) 
\to \Ext^1(\mcO(-C), \mcO_{C}(-2)), 
\]
which proves the claim. 
By a standard diagram chasing, 
we can see that $E_{2}$ fits into the exact sequence 
\[
0 \to I_{p}(-C) \to E_{2} \to \mcO_{C}(-2) \to 0 
\]
for some point $p \in C$. 
Combining with Lemma \ref{lem:ideal}, 
we conclude that 
\begin{align*}
&\Exc(\xi_{1}^+)
=\bigcup_{p \in \hatC^2} 
\bP(\Ext^1(\mcO_{C}(-2), I_{p}(-C))), \\
&\Exc(\xi_{1}^-)
=\bigcup_{p \in \hatC^2} 
\bP(\Ext^1(I_{p}(-C), \mcO_{C}(-2))), 
\end{align*}
which are $\bP^2$-bundle, 
$\bP^1$-bundle 
over $\hatC^2$, respectively. 
The diagram 
\[
M^1(1, 0, -3) \xrightarrow{\xi_{1}^-} 
M^{1, 2}(1, 0, -3) \xleftarrow{\xi_{1}^+} 
M^2(1, 0, -3) 
\]
is a family of standard flips 
parametrized by $\hatC^2$. 

Finally, let us consider the morphism 
$\xi^+_0 \colon M^1(1, 0, -3) \to \Hilb^3(\bC^2)$. 
There are two types of objects in $M^1(1, 0, -3)$: 
\begin{enumerate}
\item an ideal sheaf $I_{Z}$, 
where $Z \subset \hatC^2$ 
is a  length $3$ zero dimensional subscheme 
with $Z \cap C=\{\pt\}$. 

\item an object 
$E_{1} \in \bP(\Ext^1(I_{p}(-C), \mcO_{C}(-2)))$, 
where $p \in \hatC^2$. 
\end{enumerate}

Let us consider a sheaf $E_{1}$ of type (2). 
By a computation similar as above, 
we can see that 
$\hom(E_{1}, \mcO_{C}(-1)) \leq 2$. 
Again, by a simple diagram chasing, 
we have the following possibilities: 
\begin{itemize}
\item When $\hom(E_{1}, \mcO_{C}(-1))=1$, 
$E_{1}$ fits into a exact sequence 
\[
0 \to I_{p, q}(-C) \to E_{1} \to \mcO_{C}(-1) \to 0 
\]
for some $q \in C$ with $q \neq p$. 

\item When $\hom(E_{1}, \mcO_{C}(-1))=2$, 
$E_{1}$ fits into an exact sequence 
\[
0 \to \mcO(-2C) \to E_{1} \to \mcO_{C}(-1)^{\oplus 2} \to 0. 
\]
\end{itemize}

As a summary, we list up the properties of 
the morphism $\xi_{0}^+ \colon M^1(1, 0, -3) \to \Hilb^3(\bC^2)$: 
\begin{itemize}
\item We have 
$\xi_{0}^+(\Exc(\xi_{0}^+)) 
=M^1(1, 0, -2) \cap M^2(1, 0, -2) 
=M^{1, 2}(1, 0, -2)$, 
which has a single singular point 
$o \in M^{1, 2}(1, 0, -2)$. 
\item For a point $p \in M^{1, 2}(1, 0, -2)$, 
the fiber of $\xi_{0}^+$ is given as 
\[
(\xi_{0}^+)^{-1}(p) \cong 
\begin{cases}
F_{o}:=\bP^2 \quad (p=o) \\
F_{p}:=\bP^1 \quad (p \neq o). 
\end{cases}
\]

\item The normal bundles of fibers of $\xi_{0}^+$ are given as 
\begin{align*}
&\mcN_{F_{o}/M^1(1, 0, -3)} \cong \Omega_{\bP^2}(-1)^{\oplus 2}, \\
&\mcN_{F_{p}/M^1(1, 0, -3)} \cong \mcO_{\bP^1}(-1) \oplus \mcO_{\bP^1}^{\oplus 4} 
\quad (p \neq o). 
\end{align*}

\item The moduli space $M^{1}(1, 0, -3)$ 
is isomorphic to 
$\Bl_{M^{1, 2}(1, 0, -2)}\Hilb^3(\bC^2)$. 
\end{itemize}

Using the semi-orthogonal decompositions 
for standard flips and blow-ups 
of codimension two Cohen-Macaulay subschemes 
(cf. \cite[3.1.2]{jc18}), 
we have 
\begin{align*}
&\quad D^b(\Hilb^3(\hatC^2)) \\
&=\left\langle
D^b(\pt), D^b(\pt), D^b(\pt), 
D^b(\bC^2), D^b(\bC^2), 
D^b(\Hilb^2(\bC^2)), 
D^b(\Hilb^3(\bC^2))
\right\rangle.
\end{align*}

\subsubsection{}
We give the first example 
where the Grassmannian variety 
(which is not the projective space) 
appears as a fiber: 
Let us consider a non-trivial extension 
\[
0 \to \mcO(-2C) \to E \to \mcO_{C}(-2)^{\oplus 2} \to 0. 
\]
The sheaf $E$ is $2$-stable but not $1$-stable, 
and has a Chern character $(1, 0, -5)$. 
Noting that 
$\ext^1(\mcO_{C}(-2), \mcO(-2C))=4$, 
we have 
$\Gr(2, 4) \subset M^2(1, 0, -5)$ 
and it is contracted by the morphism 
$\xi_{1}^+ \colon M^2(1, 0, -5) \to M^{1, 2}(1, 0, -5)$.


\end{document}